\documentclass[reqno]{amsart}

\usepackage{amsthm}
\usepackage{amscd}
\usepackage{amsfonts}
\usepackage{amsmath}
\usepackage{amssymb}
\usepackage{amsrefs}
\usepackage{mathrsfs}
\usepackage{enumitem}
\usepackage{multirow}
\usepackage{verbatim}
\usepackage{url}
\usepackage{graphicx}

\usepackage{anyfontsize}
\usepackage{t1enc}

\usepackage{microtype}
\usepackage{hyperref}

\newtheorem{thm}{Theorem}[section]
\newtheorem{lem}[thm]{Lemma}
\newtheorem{prop}[thm]{Proposition}
\newtheorem{cor}[thm]{Corollary}

\newtheorem{ques}[thm]{Question}

\theoremstyle{definition}
\newtheorem{defn}[thm]{Definition}

\newcommand{\bC}{{\mathbb{C}}}
\newcommand{\bN}{{\mathbb{N}}}
\newcommand{\bR}{{\mathbb{R}}}

\newcommand{\A}{{\mathcal{A}}}
\newcommand{\B}{{\mathcal{B}}}

\newcommand{\D}{{\mathcal{D}}}

\renewcommand{\H}{{\mathcal{H}}}

\newcommand{\M}{{\mathcal{M}}}

\newcommand{\U}{{\mathcal{U}}}
\newcommand{\X}{{\mathcal{X}}}

\renewcommand{\phi}{\varphi}

\newcommand{\fA}{{\mathfrak{A}}}

\newcommand{\fM}{{\mathfrak{M}}}

\newcommand{\qand}{\quad\text{and}\quad}
\newcommand{\qqand}{\qquad\text{and}\qquad}

 \newcommand{\cconv}{\overline{\mathrm{conv}}}
\newcommand{\conv}{\mathrm{conv}}

\newcommand{\tr}{\mathrm{tr}}
\newcommand{\diag}{\mathrm{diag}}
\newcommand{\pt}{\mathrm{pt}}

\author{Xavier Mootoo}\address[Mootoo]
{Department of Mathematics and Statistics\\ York  University\\ Toronto, Ontario M3J 1P3\\ CANADA}
\email{xmootoo@gmail.com}

\author{Paul Skoufranis}\address[Skoufranis]
{Department of Mathematics and Statistics\\ York  University\\ Toronto, Ontario M3J 1P3\\ CANADA}
\email{pskoufra@yorku.ca}


\thanks{Research of the first author was supported by an NSERC (Canada) Undergraduate Student Research Award.}
\thanks{Research of the second author was supported by NSERC (Canada) grant RGPIN-2017-05711.}

\title[Joint Majorization in Continuous Matrix Algebras]{Joint Majorization in Continuous Matrix Algebras}


\begin{document}

\begin{abstract}
Various notions of joint majorization are examined in continuous matrix algebras. The relative strengths of these notions are established via proofs and examples.  In addition, the closed convex hulls of joint unitary orbits are completely characterized in continuous matrix algebras via notions of joint majorization.  Some of these characterizations are extended to subhomogeneous C$^*$-algebras.
\end{abstract}

\maketitle

\section{Introduction} 

The notion of majorization of one self-adjoint $n \times n$ matrix by another appears in many different results in mathematics.  For example, the classical theorem of Schur and Horn \cites{S1923, H1954} states that a diagonal matrix $D$ is majorized by a self-adjoint matrix $B$ if and only if a unitary conjugate of $B$ has the same diagonal as $D$.  Many different characterizations of majorization of one self-adjoint matrix $A$ by another $B$ can be found in \cite{A1989}.  Some examples include $A$ being in the convex hull of the unitary orbit of $B$, the eigenvalues of $A$ being controlled by the eigenvalues of $B$ via certain inequalities, tracial inequalities between functions of $A$ and $B$, and doubly stochastic matrices relating the eigenvalues of $A$ and $B$. Recently, the notion of majorization and its connection to closed convex hulls of unitary orbits was extended to all C$^*$-algebras in  \cite{ng2018majorization} using tracial weights.  

Majorization has more uses in mathematics than can be listed here.  For example, the recent book on quantum information theory \cite{W2018} devotes an entire chapter to majorization of self-adjoint matrices.  In particular, a self-adjoint matrix $A$ is majorized by a self-adjoint matrix $B$ if and only if there is a mixed-unitary quantum channel that maps $B$ to $A$.  Furthermore, \cite{N1999} describes the necessary and sufficient conditions for the existence of entanglement transformations in terms of majorization, and in \cite{N2000} majorization plays a role in characterizing the class of probability distributions appearing in the decomposition of a density matrix.

A ``multivariate majorization'' often called \emph{joint majorization} occurs by generalizing the notion of majorization from self-adjoint matrices to tuples of commuting self-adjoint matrices.  Within the context of quantum information theory, joint majorization determines the possible values when commuting self-adjoint matrices are sent through the same quantum channel.  Most of the characterizations of majorization from \cite{A1989} can be extended to the joint majorization setting and can be found in \cite{peria2005weak}.  It is worthwhile to note that joint majorization also plays a role in geometry and statistics \cite{KR1983} and in characterizing various classes of economic disparity indices.  Joint majorization has also been studied in the context of II$_1$ factors in  \cite{AM2008}.

In this paper, we will examine joint majorization in the setting of continuous matrix algebras.  Determining which notions of matrix majorization extend continuously could be useful in applications where the continuous evolution of majorization is of interest. Including this introduction, this paper has 7 sections summarized as follows.

In Section \ref{sec:defn}, various notions of joint majorization are considered in continuous matrix algebras.  We investigate all possible extensions of known characterizations of joint majorization in this setting and show the general order of strength of these characterizations.  Section \ref{sec:counter} demonstrates many of these characterizations are of distinct strength through some enlightening examples.

In Section \ref{sec:finite-dimensional}, joint majorization is examined in  finite dimensional C$^*$-algebras.  This is a stepping stone to proving a main theorem in Section \ref{sec:main} which proves the equivalence of several forms of joint majorization in any continuous matrix algebra (the notation of this theorem is explained in Section \ref{sec:defn}).

\begin{thm} \label{main_theorem}
Let $\X$ be a compact metric space and let $\vec{A} = (A_1, \dots, A_m)$ and $\vec{B} = (B_1, \dots, B_m)$ be abelian families in $C(\X, M_n(\bC))$. Then the following are equivalent:
\begin{enumerate}
\item $\vec{A} \prec_{\pt} \vec{B}$.
\item $\vec{A} \prec_{\tr} \vec{B}$
\item $\vec{A} \prec_{\tr, \pt} \vec{B}$
\item $\vec{A} \in \cconv(\U(\vec{B}))$.
\end{enumerate}
\end{thm}

In Section \ref{sec:subhomo}, Theorem \ref{main_theorem} is extended to characterize closed convex hulls of unitary orbits in any subhomogeneous C$^*$-algebra.  Although this result subsumes a good portion of Theorem \ref{main_theorem}, we present the material as such as there are alternate characterizations of majorization in continuous matrix algebras, many readers may only be interested in Theorem \ref{main_theorem}, and the proof of Theorem \ref{main_theorem} is essential to the proofs in Section \ref{sec:subhomo}.  Section \ref{sec:bonus} concludes with some open questions and interesting examples.

\section{Definitions and Notation in Majorization}
\label{sec:defn}

In this section we will begin with a preliminary analysis of the objects of study in this paper and various possible definitions of majorization for $m$-tuples of operators will be introduced and discussed.  In particular, this section will be completed with Proposition \ref{prop:elementary-equivalences} characterizing which properties are immediately implied by other properties.  

In this section and throughout the paper,  $\M_n$ will denote the $n \times n$ matrices with complex entries, $\D_n$ will denoted the diagonal $n \times n$ matrices, $\diag(a_1, \ldots, a_n)$ will denote the $n \times n$ matrix with $a_1,\ldots, a_n$ appearing along the diagonal, $\tr$ will denote the normalized trace on $\M_n$, $\M_{n \times m}$ will denote the $n \times m$ matrices with complex entries, $A^T$ will denote the transpose of a matrix $A$, $C(\X, \fA)$ will denote the continuous functions from a Hausdorff space $\X$ to a C$^*$-algebra $\fA$, and $C_b(\X, \fA)$ will denote the bounded elements of $C(\X, \fA)$. Moreover, given operators $A_1, \ldots, A_m$ in a C$^*$-algebra $\fA$, $C^*(A_1, \ldots, A_m)$ will be used to denote the C$^*$-subalgebra generated by $A_1, \ldots, A_m$, and all norms will be the operator norm unless otherwise specified.   Furthermore, to simplify many discussions in the paper, we make the following definitions.

\begin{defn} \label{prob_vec}
A vector $\vec{t} = (t_1,\ldots, t_n) \in \bR^n$ is said to be a \emph{probability vector} if $t_i \geq 0$ for all $1 \leq i \leq n$ and $\sum_{i = 1}^n t_i =1$.
\end{defn}

The main objects of study in this paper revolve around understanding the following objects in C$^*$-algebras.
 
\begin{defn} \label{unit_orbit}
Let $\fA$ be a unital C$^*$-algebra.  An $m$-tuple $\vec{A} = (A_1, \ldots, A_m) \in \fA^m$ is said to be an \emph{abelian family} if $A_k$ is self-adjoint and $A_k A_j = A_j A_k$ for all $1\leq j,k \leq m$.  The \emph{(joint) unitary orbit of} $\vec{A}$ is
\begin{align*}
    \mathcal{U}(\vec{A}) = \{ (U^* A_1U, \ldots, U^*A_mU) \, \mid \, U \in \fA \text{ is unitary}\} \subseteq \fA^m.
\end{align*}
The \emph{convex hull of  $\mathcal{U}(\vec{A})$} will be denoted by $\conv( \mathcal{U}(\vec{A}))$.  In particular
\[
\conv( \mathcal{U}(\vec{A})) = \left\{\left. \sum^k_{i=1} t_i \vec{C}_i \, \right| \, k \in \bN, \{\vec{C}_i\}^k_{i=1} \subseteq \mathcal{U}(\vec{A}), \vec{t} \in \bR^k \text{ a probability vector}  \right\}.
\]
Furthermore, $\cconv( \mathcal{U}(\vec{A}))$ will be used to denote the norm closure of $ \conv(\mathcal{U}(\vec{A}))$ in $\fA^m$.
\end{defn}

The goal of this paper is to attempt to characterize $\conv( \mathcal{U}(\vec{A}))$ and $\cconv( \mathcal{U}(\vec{A}))$ for specific C$^*$-algebras (i.e. the continuous matrix algebras).  Such characterizations are often called joint majorization as they involve one $m$-tuple being `larger' than another $m$-tuple.

Given an abelian family $(A_1,\ldots, A_m)$, $f(A_1,\ldots, A_m)$ is well-defined for any continuous function $f : \bR^m \to \bR$.  Consequently, the following form of joint majorization is well-defined in any C$^*$-algebra.

\begin{defn}
Let $\fA$ be a unital C$^*$-algebra.  Recall a \emph{tracial state on $\fA$} is a unital positive linear functional $\tau : \fA \to \bC$ such that $\tau(AB) = \tau(BA)$ for all $A, B \in \fA$.

Let $\vec{A} = (A_1, \ldots, A_m)$ and $\vec{B} = (B_1, \ldots, B_m)$ be abelian families in $\fA$.  It is said that $\vec{A}$ is \emph{tracially majorized} by $\vec{B}$, denoted $\vec{A} \prec_{\tr} \vec{B}$, if for every tracial state $\tau : \fA \to \bC$ and every continuous convex function $f : \bR^m \to \bR$ we have that
\[
\tau(f(A_1,\ldots, A_m)) \leq \tau(f(B_1,\ldots, B_m)).
\]
\end{defn}

It is important to note for future use that the set of tracial states on a C$^*$-algebra is a convex, weak$^*$-compact set.  

Although tracial majorization is well-defined in any C$^*$-algebra, there are other characterizations of joint majorization for matrix algebras that rely on the following collection of matrices that appear often in statistics.

\begin{defn} \label{ds_n}
Let $X = [a_{i,j}] \in \M_n$.
It is said that $X$ is \emph{doubly stochastic} if
\begin{enumerate}
    \item $a_{i,j} \geq 0$ for all $1 \leq i,j \leq n$,  
    \item $\sum_{i = 1}^n a_{i,j} = 1$ for all $1 \leq j \leq n$, and 
    \item $\sum_{j = 1}^n a_{i,j} = 1$ for all $1 \leq i \leq n$. 
\end{enumerate}
The set of $n \times n$ doubly stochastic matrices will be denoted by $DS_n$.

It is said that $X$ is \emph{unistochastic} if there exists an $n \times n$ unitary matrix $[u_{ij}]$ such that $a_{ij} = |u_{ij}|^2$ for all $1 \leq i,j \leq n$. The set of unistochastic matrices will be denoted by $US_n$.
\end{defn}
It is not difficult to verify that $DS_n$ is a compact, convex subset of $\M_n$ and that $US_n \subseteq DS_n$ for all $n \in \bN$.  However, it is well-known that not every doubly stochastic matrix is unistochastic.  In particular, it is not difficult to verify that
\[
\frac{1}{2} \begin{bmatrix}
1 & 0 & 1 \\
1 & 1 & 0 \\
0 & 1 & 1
\end{bmatrix}
\]
is doubly stochastic but not unistochastic (see Proposition \ref{prop:counter-birkhoff-von-neumann} for a more general argument).

Using doubly stochastic matrices, we arrive at the following notion of joint majorization for matrices.

\begin{defn}[\cite{KR1983}]  \label{joint_maj}
Let $\vec{A} = (A_1, \dots, A_m)$ and $\vec{B} = (B_1, \dots, B_m)$ be abelian families in $\M_n$.  Thus $\{A_1,\ldots, A_m\}$ and $\{B_1, \ldots, B_m\}$ are simultaneously diagonalizable sets of self-adjoint matrices.  So for $1 \leq j \leq m$ there exists diagonal matrices $D_j$ and $D'_j$ and unitaries $U, V \in \M_n$ such that
\[
A_j = U^* D_j U \qand B_j = V^* D'_j V
\]
for all $1 \leq j \leq m$.  

For all $1\leq j \leq m$ let  $\vec{\lambda}(A_j)$ and $\vec{\lambda}(B_j)$ denote the eigenvalue vectors for $A_j$ and $B_j$ as they appear along the diagonals of $D_j$ and $D'_j$ respectively, and let $A$ and $B$ be the $n \times m$ matrices whose $j^{\mathrm{th}}$ columns are $\vec{\lambda}(A_j)$ and $\vec{\lambda}(B_j)$ respectively for $1\leq j \leq m$.  It is said that $\vec{A}$ is \emph{(doubly stochastic) majorized} by $\vec{B}$, denoted $\vec{A} \prec \vec{B}$, if there exists a $X \in DS_n$ such that $XB=A$.
\end{defn}

The connection between the tracial and doubly stochastic versions of joint majorization and convex hulls of joint unitary orbits in matrix algebras is known.

\begin{thm}[\cite{D1999}*{Corollary 3.5}, \cite{KR1983}, \cite{peria2005weak}*{Proposition 4.2, Theorem 4.5, Proposition 4.9}]
\label{thm:matrix-only-result}
Let $\vec{A} = (A_1, \dots, A_m)$ and $\vec{B} = (B_1, \dots, B_m)$ be abelian families in $\M_n$.   The following are equivalent:
\begin{enumerate}
\item $\vec{A} \prec \vec{B}$.
\item $\vec{A} \prec_{\tr} \vec{B}$.
\item There exists a unital, trace-preserving, (completely) positive map 
\[
\Phi : C^*(A_1, \ldots, A_m) \to C^*(B_1, \ldots, B_m)
\]
such that $\Phi(A_k) = B_k$ for all $1 \leq k \leq m$.  
\item $\vec{A} \in \conv(\U(\vec{B}))$.
\item $\vec{A} \in \cconv(\mathcal{U}(\vec{B}))$.
\end{enumerate}
\end{thm}

The proof that property (5) is equivalent to the other properties in Theorem \ref{thm:matrix-only-result} does not usually appear in the literature.  A proof of its equivalence is contained in the proof that (7) implies (2) in Proposition \ref{prop:elementary-equivalences}.

Before examining the extent to which Theorem \ref{thm:matrix-only-result} generalizes to continuous matrix algebras, one must first examine the possible ways the notions of majorization generalize.  This immediately causes some slight difficulties in describing Definition \ref{joint_maj} in $C(\X, \M_n)$ since although abelian families in von Neumann algebras  are always simultaneously diagonalizable (see \cite{K1984}),  abelian families in $C(\X, \M_n)$ need not be simultaneously diagonalizable  (see \cite{GP1984}).  Assuming our abelian families have the following property, it is then possible to discuss a continuous version of doubly stochastic majorization.

\begin{defn}  
Let $\X$ be a Hausdorff space.  An abelian family $(A_1, \dots, A_m)$ in $C(\X, \M_n)$ is said to be \emph{continuously simultaneously diagonalizable} (CSD) if there exists a unitary $U \in C(\X, \M_n)$ and $D_1, \dots, D_m \in C(\X, \D_n)$ such that
\[
A_i = UD_iU^*
\]
for all $1 \leq i \leq m$.
\end{defn}

\begin{defn}\label{defn:joint-cont-DS-major}
Let $\X$ be a Hausdorff space and let $\vec{A} = (A_1, \dots, A_m)$ and $\vec{B} = (B_1, \dots, B_m)$ be CSD families in $C(\X, \M_n)$. Thus for $1 \leq j \leq m$ there exists $D_j, D'_j \in C(\X, \D_n)$ and unitaries $U, V \in C(\X, \M_n)$ such that
\[
A_j = U^* D_j U \qand B_j = V^* D'_j V
\]
for all $1 \leq j \leq m$.  

For all $1\leq j \leq m$ let $\vec{\lambda}(A_j)$ and $\vec{\lambda}(B_j)$ denote the vectors of continuous functions appearing along the diagonals of $D_j$ and $D'_j$ (in the order that they appear) respectively, and let $A,B \in C(\X, \M_{n \times m})$ be such that the $j^{\mathrm{th}}$ columns of $A$ and $B$ are $\vec{\lambda}(A_j)$ and $\vec{\lambda}(B_j)$ respectively for $1 \leq j \leq m$.  It is said that $\vec{A}$ is \emph{continuously (doubly stochastic) majorized} by $\vec{B}$, denoted $\vec{A} \prec_c \vec{B}$, if there exists a  $X \in C(\X, DS_n)$ such that $XB=A$.
\end{defn}

Of course, a generalization of doubly stochastic majorization that does not require a CSD family can be easily obtained by considering the question of majorization pointwise.

\begin{defn}\label{defn:pt-DS-major}
Let $\X$ be a Hausdorff space and let $\vec{A} = (A_1, \dots, A_m)$ and $\vec{B} = (B_1, \dots, B_m)$ be abelian families in $C(\X, M_n)$.  It is said that $\vec{A}$ is \emph{pointwise (doubly stochastic) majorized} by $\vec{B}$, denoted $\vec{A} \prec_{\pt} \vec{B}$,  if $A(x) \prec B(x)$ for each $x \in \X$.
\end{defn}

While on the topic of pointwise majorization, although the notion of tracial majorization clearly exists in $C(\X, \M_n)$, a pointwise version can easily be discussed.

\begin{defn}
Let $\X$ be a Hausdorff space and let $\vec{A} = (A_1, \dots, A_m)$ and $\vec{B} = (B_1, \dots, B_m)$ be abelian families in $C(\X, M_n)$.  It is said that $\vec{A}$ is \emph{pointwise tracially majorized} by $\vec{B}$, denoted $\vec{A} \prec_{\tr, \pt} \vec{B}$, if $A(x) \prec_{\tr} B(x)$ for each $x \in \X$.
\end{defn}

With the forms of joint majorization in $C(\X, \M_n)$ defined above, certain equivalences from Theorem \ref{thm:matrix-only-result} can be deduced.

\begin{prop}\label{prop:elementary-equivalences}
Let $\X$ be a compact Hausdorff space and let $\vec{A} = (A_1, \dots, A_m)$ and $\vec{B} = (B_1, \dots, B_m)$ be abelian families in $C(\X, M_n)$.  Consider the following properties:
\begin{enumerate}
\item Assuming $\vec{A}$ and $\vec{B}$ are CSD, $\vec{A} \prec_c \vec{B}$.
\item $\vec{A} \prec_{\pt} \vec{B}$.
\item $\vec{A} \prec_{\tr} \vec{B}$.
\item $\vec{A} \prec_{\tr, \pt} \vec{B}$.
\item Assuming $\vec{A}$ and $\vec{B}$ are CSD, for each $x\in \X$ there exists a unital, trace-preserving, (completely) positive map $\Psi_x : \D_n \to \D_n$ such that if
\[
\Psi : C(\X, \D_n) \to C(\X, \D_n)
\]
is defined by
\[
\Psi(T)(x) = \Psi_x(T(x))
\]
for every $T \in C(\X, \D_n)$ and $x \in \X$, then $\Psi$ is well-defined, and there exists unitaries $U, V \in C(\X, \M_n)$ such that 
\[
U C^*(A_1,\ldots, A_m) U^* \subseteq C(\X, \D_n)
\]
and if $\Phi : C^*(A_1,\ldots, A_m) \to C(\X, \M_n)$ is defined by
\[
\Phi(T) = V^*\Psi(UTU^*)V
\]
for all $T \in C^*(A_1,\ldots, A_m)$, then $\Phi(A_k) = B_k$ for all $1 \leq k \leq m$ (and $\Phi$ is a unital (completely) positive map that preserves every trace on $C(\X, \M_n)$).
\item $\vec{A} \in \conv(\U(\vec{B}))$.
\item $\vec{A} \in \cconv(\U(\vec{B}))$.
\end{enumerate}
Then (6) $\implies$ (1)  $\Longleftrightarrow$ (5) $\implies$ (2) and  (6) $\implies$ (7)  $\implies$ (2) $\Longleftrightarrow$ (4) $\Longleftrightarrow$ (3).
\end{prop}
\begin{proof}
Clearly (6) $\implies$ (7) and clearly (1) $\implies$ (2).  Moreover (2) $\Longleftrightarrow$ (4) by Theorem \ref{thm:matrix-only-result}.  

\begin{proof}[Proof of (4) $\Longleftrightarrow$ (3)]
Since for each $x \in \X$ the map $\tr_x : C(\X, \M_n) \to \bC$ defined by
\[
\tr_x(T) = \tr(T(x))
\]
is a tracial state on $C(\X, \M_n)$, it is trivial to see that (3) implies (4).  Moreover, since $\M_n$ has exactly one trace, $\{\tr_x \, \mid \, x \in X\}$ are the extreme points of the set of all tracial states on $C(\X, \M_n)$ when $\X$ is a normal topological space.  Consequently, by considering the weak$^*$ topology on the tracial states, it easily follows from the Krein-Milman Theorem that (4) implies (3).
\end{proof}

\begin{proof}[Proof of (5) $\implies$ (1)]
Let $U$, $V$, $\Psi_x$ for $x \in \X$, $\Psi$, and $\Phi$ be as in the statement of (5).  By the assumption of (5), $\vec{B}$ is simultaneously diagonalizable via $V$.  Let $A$, $B$, $D_j$, $D'_j$, $\vec{\lambda}(A_j)$, and $\vec{\lambda}(B_j)$ for $1\leq j \leq m$ be as in Definition \ref{defn:joint-cont-DS-major}.  Hence $\Psi(D'_j) = D_j$ for all $1 \leq j \leq m$.  

For each $1 \leq i \leq n$ let $E_{i,i} \in C(\X, \D_n)$ be the matrix with constant function 1 in the $(i,i)$-entry and zero everywhere else and let $X : \X \to \M_n$ be such that the $j^{\mathrm{th}}$ column of $X$ is the diagonal entries in the order they appear of $\Psi(E_{i,i})$ for all $1\leq i \leq n$.  Hence $X \in C(\X, \M_n)$.  

To see that $X \in C(\X, DS_n)$, first notice that since $E_{i,i}$ is a positive operator for all $1\leq i \leq m$, $\Psi(E_{i,i})$ is positive and thus has non-negative diagonal entries.  Hence the entries of $X(x)$ are non-negative for all $x \in \X$.  Moreover, since 
\[
\tr_x(\Psi(E_{i,i})) = \tr(\Psi_x(E_{i,i}(x))) = \tr(E_{i,i}(x))
\]
for all $x \in \X$, we see summing the entries of each column of $X(x)$ yields 1 for all $x \in \X$.  Finally, since $\Psi_x$ is unital for all $x \in \X$,
\[
I_n = \Psi_x\left(\sum^n_{i=1} E_{i,i}(x)  \right) =\sum^n_{i=1}  \Psi_x\left(E_{i,i}(x)  \right)
\]
and thus we see summing the entries of each row of $X(x)$ yields 1 for all $x \in X$.  Hence $X \in C(\X, DS_n)$.

Since $\Psi_x(D'_j(x)) = D_j(x)$ for all $x \in \X$, we immediately see that $X(x) B(x) = A(x)$ for all $x \in \X$. Hence $XB = A$ so $\vec{A} \prec_c \vec{B}$.  
\end{proof}

\begin{proof}[Proof of (1) $\implies$ (5)]
Assume $\vec{A}$ and $\vec{B}$ are CSD and $\vec{A} \prec_c \vec{B}$.  Let $X \in C(\X, DS_n)$, $U$, $V$, $A$, $B$, $D_j$, $D'_j$, $\vec{\lambda}(A_j)$, and $\vec{\lambda}(B_j)$ for $1\leq j \leq m$ be as in Definition \ref{defn:joint-cont-DS-major}.  To prove (5), it suffices for each $x \in \X$ to construct a unital, trace-preserving, (completely) positive map $\Psi_x : \D_n \to \D_n$ such that $\Psi_x(D'_j(x)) = D_j(x)$ for all $1 \leq j \leq m$ and 
\[
x \mapsto \Psi_x(T(x))
\]
is continuous for all $T \in C(\X, \D_n)$.  

Using the same notation as in (5) $\implies$ (1), define $\Psi_x : \D_n \to \D_n$ so that $\Psi_x(E_{i,i}(x))$ is the diagonal matrix whose entries are the $i^{\mathrm{th}}$ column of $X(x)$ for all $x \in \X$.  By similar arguments to those used in (5) $\implies$ (1), $\Psi_x$ is unital, trace-preserving, and positive for all $x \in \X$ (and automatically completely positive having $\D_n$ as a co-domain).  Finally, if $T \in C(\X, \D_n)$ is such that $T = \diag(f_1, \ldots, f_n)$ with $f_1, \ldots, f_n \in C(\X)$, then $\Psi(T) = \diag(g_1, \ldots, g_n)$ where $\vec{g} = X \vec{f}$ (with $\vec{f}$ and $\vec{g}$ column vectors).  Therefore, since $X  \in C(\X, DS_n)$, we see that $g_1, \ldots, g_n \in C(\X)$ so $\Psi(T) \in C(\X, \D_n)$ for all $T \in C(\X, \D_n)$ as desired.
\end{proof}

\begin{proof}[Proof of (6) $\implies$ (1)]
Suppose $\vec{A}$ and $\vec{B}$ are CSD and $\vec{A} \in \conv(\U(\vec{B}))$.  Since $\vec{A}$ and $\vec{B}$ are CSD, for $1 \leq j \leq m$ there exists diagonal matrices $D_j, D'_j \in C(X, \M_n)$ and unitaries $U, V \in C(X, \M_n)$ such that
\[
A_j = U^* D_j U \qand B_j = V^* D'_j V
\]
for all $1 \leq j \leq m$.  Let $A,B \in C(\X, \M_{n \times m})$ be as in Definition \ref{defn:joint-cont-DS-major}.

Since $\vec{A} \in \conv(\U(\vec{B}))$, there exists a $k \in \bN$, unitaries $U_1, \ldots, U_k \in C(\X, \M_n)$, and a probability vector $\vec{t} \in \bR^k$ such that
\[
A_j = \sum^k_{i=1} t_i U^*_i B_j U_i
\]
for all $1 \leq j \leq m$.    Hence
\begin{align}
D_j = \sum^k_{i=1} t_i W^*_i D'_j W_i  \label{eq:diagonal-convex-hull-to-DS}
\end{align}
for all $1 \leq j \leq m$ where $W_i = VU_iU^*$ are unitary for all $1\leq i \leq k$. Let
\[
X = \sum^k_{i=1} t_i W^*_iW_i.
\]
Clearly $X \in C(\X, \M_n)$.  Moreover, since $\vec{t}$ is a probability vector and $W_i$ is a unitary operator for all $1\leq i \leq m$, we see that 
\[
X \in C(\X, \conv(US_n)) \subseteq C(\X, DS_n).
\] Using equation (\ref{eq:diagonal-convex-hull-to-DS}), a simple computation shows that $XB = A$.  Hence $\vec{A} \prec_c \vec{B}$.
\end{proof}

\begin{proof}[Proof of (7) $\implies$ (2)]
Suppose $\vec{A} \in \cconv(\U(\vec{B}))$.  Clearly this implies that 
\[
\vec{A}(x) \in \cconv(\U(\vec{B}(x)))
\]
for all $x \in \X$.  Thus, to prove the validity of (2), it suffices to consider the case when $\X = \{1\}$ so that $C(\X, \M_n) = \M_n$.    In this case $\vec{A}$ and $\vec{B}$ are simultaneously diagonalizable and thus for $1 \leq j \leq m$ there exists diagonal matrices $D_j, D'_j \in  \M_n$ and unitaries $U, V \in \M_n$ such that
\[
A_j = U^* D_j U \qand B_j = V^* D'_j V
\]
for all $1 \leq j \leq m$.   Let $A, B \in \M_{n \times m}$ be as in Definition \ref{joint_maj}.

Since $\vec{A} \in \cconv(\U(\vec{B}))$, there exists a sequence $(\ell_k)_{k\geq 1} \subseteq \bN$, $\ell_k$-tuples of unitaries $U_{1,k}, \ldots, U_{\ell_k,k} \in \M_n$, and probability vectors $\vec{t}_k \in \bR^{\ell_k}$ for all $k$  such that
\[
A_j = \lim_{k \to \infty} \sum^{\ell_k}_{i=1} t_{i,k} U^*_{i,k} B_j U_{i,k}
\]
for all $1 \leq j \leq m$.  Thus
\[
D_j = \lim_{k \to \infty} \sum^{\ell_k}_{i=1} t_{i,k} W^*_{i,k} D'_j W_{i,k}
\]
for all $1 \leq j \leq m$ where $W_{i,k} = VU_{i,k}U^*$ are unitary for all $1\leq i \leq \ell_k$ and $k \in \bN$.  Thus, if
\[
X_k = \sum^{\ell_k}_{i=1} t_{i,k} W^*_{i,k}W_{i,k}
\]
for all $k \in \bN$, then, as in the proof that (5) implies (1), $(X_k)_{k\geq 1}$ is a sequence of doubly stochastic matrices such that
\[
A = \lim_{k \to \infty} X_k B.
\]
Since the set of doubly stochastic matrices is compact, there exists a subsequence of $(X_k)_{k\geq 1}$ that converges to some doubly stochastic matrix $X$.  Hence $A = XB$ so $\vec{A} \prec_{\pt} \vec{B}$. 
\end{proof}
As all of the desired implications have been shown, the proof is complete.
\end{proof}

In Proposition \ref{prop:elementary-equivalences}, the completely positive map form of joint majorization is slightly dis-satisfactory when compared to the completely positive map form of joint majorization in Theorem \ref{thm:matrix-only-result}.  However, as the equivalence of (1) and (5) in Proposition \ref{prop:elementary-equivalences} demonstrates, this appears to be the correct notion.  The completely positive map form of joint majorization in Theorem \ref{thm:matrix-only-result} works due to the fact that $\M_n$ is an injective von Neumann algebra (and $C^*(B_1,\ldots, B_n)$ is an injective von Neumann subalgebra).  A similar completely positive map form of joint majorization in II$_1$ factors was a focus of and achieved in \cite{AM2008}.

\section{Some Examples}
\label{sec:counter}

As Proposition \ref{prop:elementary-equivalences} has demonstrated certain forms of joint majorization are implied by others, it is natural to consider whether the remaining implications hold.  To begin, does $\vec{A} \prec_{\pt} \vec{B}$ imply $\vec{A} \prec_c \vec{B}$?  In Definition \ref{defn:pt-DS-major}, one could imagine that having at each point $x \in \X$ a doubly stochastic matrix $X(x)$ such that $X(x) B(x) =A(x)$ plus the continuity of $A$ and $B$ could imply that $x \mapsto X(x)$ can be chosen continuously.  However, this is not the case.

\begin{prop}\label{prop:point-doesnt-go-to-continuous}
There exists CSD families $A_1$ and $B_1$ in $C([-1,1], \M_2)$ such that $A_1 \prec_{\pt} B_1$ yet $A_1 \nprec_c B_1$.
\end{prop}
\begin{proof} 
Let $A_1, B_1 \in C([-1,1],\M_2)$ be defined by
\begin{align*}
A_1(x) &= \begin{bmatrix}
x & 0 \\
0 & -x
\end{bmatrix} & B_1(x) &=  \begin{bmatrix}
|x| & 0 \\ 0 & -|x|
\end{bmatrix}.
\end{align*}
Clearly $A_1$ and $B_1$ are CSD families (as they are already diagonal matrices at each point).  Moreover, if $A$ and $B$ are as in Definition \ref{defn:joint-cont-DS-major}, then
\[
A = \begin{bmatrix}
x \\ -x
\end{bmatrix} \qqand B = \begin{bmatrix}
|x| \\ -|x|
\end{bmatrix}
\]
for all $x \in [-1,1]$.  

To see that $A_1 \prec_{\pt} B_1$, define $X:[-1,1] \to \M_2$ by
\[
X(x) = \begin{cases}\begin{bmatrix} 1 & 0 \\ 0 & 1 \end{bmatrix}&\textnormal{if } x \geq 0 \\ \ \\  \begin{bmatrix} 0 & 1 \\ 1 & 0 \end{bmatrix} &\textnormal{ if } x < 0\end{cases}
\]
for all $x \in [-1,1]$. Clearly $X(x) \in DS_2$ for all $x \in [-1,1]$.  Since $X(x)B(x) = A(x)$ for all $x \in [-1,1]$, we have $A_1 \prec_{\pt} B_1$.

To see that $A_1 \nprec_c B_1$, suppose to the contrary that there is a $Y \in C([-1,1], DS_2)$ such that $YB = A$.  Since $Y \in C([-1,1], DS_2)$, we can write
\[
Y(x) = \begin{bmatrix}
f(x) & 1-f(x) \\
1-f(x) & f(x)
\end{bmatrix}
\]
for some $f \in C[-1,1]$.  Since $YB = A$ for all $x > 0$ we have that
\[
\begin{bmatrix}
x \\ -x
\end{bmatrix} = A(x) = Y(x)B(x) = Y(x)\begin{bmatrix}
x \\ -x
\end{bmatrix} = \begin{bmatrix}
2x f(x) - x \\ x - 2xf(x)
\end{bmatrix}
\]
and thus $f(x) = 1$ for all $x > 0$.   Similarly, for all $x < 0$ we have that
\[
\begin{bmatrix}
x \\ -x
\end{bmatrix}= A(x) = Y(x)B(x) = Y(x) \begin{bmatrix}
-x \\ x
\end{bmatrix} =  \begin{bmatrix}
x - 2xf(x) \\ 2x f(x) - x
\end{bmatrix}
\]
and thus $f(x) = 0$ for all $x < 0$.  Since $\lim_{x \to 0+} f(x) \neq \lim_{x \to 0-} f(x)$, $f$ is not continuous at 0.   Hence $A_1 \nprec_c B_1$.
\end{proof}

The next natural question would be, does  $\vec{A}\prec_c \vec{B}$ imply $\vec{A} \in \conv(\U(\vec{B}))$?    Of course, when $n=1$ the result is trivially true as $DS_1 = \{1\}$ so $\vec{A} = \vec{B}$, and when $n = 2$ the proof is as follows.

\begin{prop} \label{M2_thm}
Let $\X$ be a Hausdorff space and let $\vec{A} = (A_1, \dots, A_m)$ and $\vec{B} = (B_1, \dots, B_m)$ be  CSD families in $C(\X, \M_2)$. If $\vec{A} \prec_c \vec{B}$, then $\vec{A} \in \conv(\U(\vec{B}))$.
\end{prop}
\begin{proof} 
Suppose $\vec{A} \prec_c \vec{B}$ and let $U$, $V$, $X$, $D_j$, and $D'_j$ for $1\leq j \leq m$ be as in Definition \ref{defn:joint-cont-DS-major}. Since $X \in C(\X, DS_2)$, there exists a $f \in C(\X, [0,1])$ such that
\[
X(x) = \begin{bmatrix} f(x) & 1-f(x) \\ 1-f(x) & f(x) \end{bmatrix}
\]
for all $x \in \X$. Define $\theta\in C(\X)$ by $\theta(x) = \arccos(\sqrt{f(x)})$ for all $x \in \X$. Thus
\[
X(x) = \begin{bmatrix} \cos^2 (\theta(x)) & \sin^2 (\theta(x))\\ \sin^2(\theta(x)) & \cos^2 (\theta(x)) \end{bmatrix}
\]
for all $x \in \X$. Define $W_1, W_2 \in C(\X, \M_2)$ by
\begin{align*}
 W_1(x) &= \begin{bmatrix} \cos(\theta(x)) & \sin(\theta(x)) \\ - \sin(\theta(x)) & \cos(\theta(x)) \end{bmatrix}, \\ W_2(x) &= \begin{bmatrix} \cos(\theta(x)) & -\sin(\theta(x)) \\  \sin(\theta(x)) & \cos(\theta(x)) \end{bmatrix},
\end{align*}
for all $x \in \X$.  Clearly $W_1$ and $W_2$ are unitary operators in $C(\X, \M_2)$.  As it is readily verified that if $D= \diag(c_1, c_2)$ then
\[
\frac{1}{2}  W_1(x)^* D  W_1(x) + \frac{1}{2} W_2(x)^* D W_2(x) = \diag(d_1, d_2)
\]
where $(d_1, d_2)^T = X(x)(c_1, c_2)^T$, it follows that
\[
\frac{1}{2}  W_1^* D'_j  W_1 + \frac{1}{2} W_2 ^* D'_j W_2  = D_j
\]
for all $1\leq j \leq m$.  Hence
\[
\frac{1}{2} (VW_1U^*)^* B_j  (VW_1U^*) + \frac{1}{2} (VW_2U^*) ^*B_j (VW_1U^*) = A_j
\]
for all $1 \leq j \leq m$. Thus $\vec{A} \in \conv(\U(\vec{B}))$. 
\end{proof}

One method for attempting to upgrade the result of Proposition \ref{M2_thm} to higher dimensional matrices would be to mimic the appropriate portion of the proof of Theorem \ref{thm:matrix-only-result}.  One such proof is found in \cite{peria2005weak}*{Proposition 4.2} where the Birkhoff-von Neumann Theorem \cite{B1946} is used to acquire the appropriate unitaries and the correct convex combination of elements of $\U(\vec{B})$ to yield $\vec{A}$.  However to attempt this proof in our context would require a continuous version of the Birkhoff-von Neumann Theorem, which does not hold; that is, not every $D \in C(\X, DS_n)$ can be written as $D = \sum_{i = 1}^k t_i P_i$ for some probability vector $\vec{t} \in \bR^k$ and permutation-valued maps $P_1, \dots, P_k \in C(\X, DS_n)$.   In fact, even increasing the set of permutation matrices to the set of unistochastic matrices (which would be enough in the proof of \cite{peria2005weak}*{Proposition 4.2}) fails.

\begin{prop}\label{prop:counter-birkhoff-von-neumann}
$\conv(C([-1, 1], US_3)) \subsetneq C([-1, 1], DS_3)$.
\end{prop}
\begin{proof}
Since every unistochastic matrix is doubly stochastic and since the doubly stochastic matrices are convex, $\conv(C([-1, 1], US_3)) \subseteq C([-1, 1], DS_3)$.

To see that $\conv(C([-1, 1], US_3)) \neq  C([-1, 1], DS_3)$, let $X \in C([-1, 1], \M_3)$ be defined by
\[
X(x) = \sin^2(x) \begin{bmatrix}1 & 0 & 0 \\ 0 & 1 & 0 \\ 0 & 0 & 1 \end{bmatrix} + \cos^2(x) \begin{bmatrix}0 & 1 & 0 \\ 0 & 0 & 1 \\ 1 & 0 & 0 \end{bmatrix} = \begin{bmatrix} \sin^2(x) & \cos^2(x) & 0 \\ 0 & \sin^2(x) & \cos^2(x) \\ \cos^2(x) & 0 & \sin^2(x) \end{bmatrix}
\]
for all $x \in [-1, 1]$.   Clearly $X(x) \in DS_3$ for all $x \in [-1, 1]$.

To see that $X \notin \conv(C([-1, 1], US_3))$, suppose to the contrary that there exists a $\ell \in \bN$, a probability vector $\vec{x} \in \bR^\ell$, and $Q_1, \dots , Q_\ell \in C([-1, 1], US_3)$ such that 
\begin{align}
X = \sum_{k = 1}^\ell t_k Q_k. \label{eq:DS-not-convex-combo}
\end{align}
Thus for each $1 \leq k \leq \ell$ there exists a unitary $U_k \in C([-1, 1], \M_3)$ such that if $Q_k(x)= [q^{(k)}_{i,j}(x)]$ and $U_k(x) = [u_{i,j}^{(k)}(x)]$ then $q_{i,j}^{(k)}(x) = |u_{i,j}^{(k)}(x)| ^2$ for all $1 \leq i,j \leq 3$ and $x \in [-1, 1]$.  Consequently
\[
\begin{bmatrix} \sin^2(x) & \cos^2(x) & 0 \\ 0 & \sin^2(x) & \cos^2(x) \\ \cos^2(x) & 0 & \sin^2(x) \end{bmatrix} = \sum_{k = 1}^\ell t_k \begin{bmatrix} |u_{1,1}^{(k)}(x)| ^2 & |u_{1,2}^{(k)}(x)| ^2 & |u_{1,3}^{(k)}(x)| ^2 \\ |u_{2,1}^{(k)}(x)| ^2 & |u_{2,2}^{(k)}(x)| ^2 & |u_{2,3}^{(k)}(x)| ^2 \\ |u_{3,1}^{(k)}(x)| ^2& |u_{3,2}^{(k)}(x)| ^2 & |u_{3,3}^{(k)}(x)| ^2\end{bmatrix}  
\]
for all $x \in [-1,1]$.  Hence 
\[
u_{1,3}^{(k)}(x)= u_{2,1}^{(k)}(x) = u_{3,2}^{(k)}(x)=0
\]
for all $x \in [-1,1]$ and $1 \leq k \leq \ell$. Therefore
\[
U_k(x) = \begin{bmatrix} u_{1,1}^{(k)}(x)& u_{1,2}^{(k)}(x)& 0\\ 0& u_{2,2}^{(k)}(x)& u_{2,3}^{(k)}(x) \\ u_{3,1}^{(k)}(x)& 0& u_{3,3}^{(k)}(x) \end{bmatrix}
\]
for all $x \in [-1, 1]$ and $1 \leq k \leq \ell$. Since $U_k$ is a unitary and thus has orthonormal rows and columns, we obtain that
\begin{align*}
     |u_{1,1}^{(k)}(x)|^2 + |u_{1,2}^{(k)}(x)|^2  &=1 \text{ and} \\
     u_{1,1}^{(k)}(x) \overline{u_{1,2}^{(k)}(x)} &=0
\end{align*}
for every $x \in [-1, 1]$.  Since $u_{1,1}^{(k)} $ and $u_{1,2}^{(k)} $ are continuous functions on $[-1, 1]$, this forces $u_{1,1}^{(k)} =0 $ or $u_{1,2}^{(k)} =0$.

In the case $u_{1,1}^{(k)} =0 $, the first column of $U_k$ forces $|u^{(k)}_{3,1}| = 1$, the first and third columns forces $u^{(k)}_{3,3} = 0$, the third column then forces $|u^{(k)}_{2,3} | = 1$, the second and third columns forces $u^{(k)}_{2,2} = 0$, and thus the second column forces $|u^{(k)}_{1,2} | = 1$.  Hence
\[
Q_k(x) =  \begin{bmatrix}0 & 1 & 0 \\ 0 & 0 & 1 \\ 1 & 0 & 0 \end{bmatrix}
\]
in this case.  Similarly, if $u_{1,2}^{(k)} =0 $, then
\[
 Q_k(x)  = \begin{bmatrix}1 & 0 & 0 \\ 0 & 1 & 0 \\ 0 & 0 & 1 \end{bmatrix}.
\]
However, equation (\ref{eq:DS-not-convex-combo}) then implies $X$ is constant, which is a clear contradiction.  Hence $X \notin \conv(C([-1, 1], US_3))$.
\end{proof}

Although a lack of a continuous version of the Birkhoff-von Neumann Theorem only implies the proof of \cite{peria2005weak}*{Proposition 4.2} does not generalize, the same example used in the proof of Proposition \ref{prop:counter-birkhoff-von-neumann} yields a negative answer to the question on whether $\vec{A} \prec_c \vec{B}$ implies $\vec{A} \in \conv(\U(\vec{B}))$.

\begin{cor}\label{cor:cont-major-does-not-imply-exact-convex-hull}
There exists CSD families $\vec{A} = (A_1, A_2)$ and $\vec{B} = (B_1, B_2)$ in $C([-1, 1], \M_3)$ such that $\vec{A} \prec_c \vec{B}$ yet $\vec{A} \notin \conv(\U(\vec{B}))$.
\end{cor}
\begin{proof}
Let $A_1, A_2, B_1, B_2 \in C([-1, 1], \M_3)$ be defined by
\begin{align*}
A_1(x) &= \begin{bmatrix}
\sin^2(x) & 0 & 0 \\ 0 & 0 & 0 \\ 0 &0  & \cos^2(x)
\end{bmatrix}   & A_2(x) &=  \begin{bmatrix}
\cos^2(x) & 0 & 0 \\ 0 & \sin^2(x) & 0 \\ 0 & 0 & 0 
\end{bmatrix} \\
B_1(x) &= \begin{bmatrix}
1 & 0 & 0 \\ 0 & 0 & 0 \\ 0 & 0 & 0 
\end{bmatrix}   & B_2(x) & = \begin{bmatrix}
0 & 0 & 0 \\ 0 & 1 & 0 \\ 0 & 0 & 0 
\end{bmatrix}
\end{align*}
for all $x \in [-1,1]$.  Clearly $\vec{A} = (A_1, A_2)$ and $\vec{B} = (B_1, B_2)$ are CSD families (as they are already diagonal matrices).  Moreover, if $A$ and $B$ are as in Definition \ref{defn:joint-cont-DS-major}, then
\[
A =\begin{bmatrix}
 \sin^2(x) & \cos^2(x) \\
 0 & \sin^2(x) \\
 \cos^2(x) & 0 
\end{bmatrix}   \qqand B = \begin{bmatrix}
1 & 0 \\
0 & 1 \\
0 & 0
\end{bmatrix}
\]
for all $x \in [-1,1]$.  

To see that $\vec{A} \prec_{c} \vec{B}$, let $X \in C([-1,1], DS_3)$ be as in Proposition \ref{prop:counter-birkhoff-von-neumann}.  As clearly $XB = A$, it is verified that $\vec{A} \prec_{c} \vec{B}$.

To see that $\vec{A} \notin \conv( \U(\vec{B}))$, suppose to the contrary that there exists a $\ell \in \bN$, a probability vector $\vec{t} \in \bR^\ell$, and unitaries $U_1, \ldots, U_\ell \in C([-1,1], \M_3)$ such that
\[
A_j = \sum^\ell_{k=1} t_k U_k^* B_j U_k
\]
for $j = 1,2$.   Write $U_k(x) = [u_{i,j}^{(k)}(x)]$ for all $1\leq k \leq \ell$ and note that
\begin{align*}
\begin{bmatrix}
\sin^2(x) & 0 & 0 \\ 0 & 0 & 0 \\ 0 &0  & \cos^2(x)
\end{bmatrix} &= \sum^\ell_{k=1} t_k \begin{bmatrix}
|u^{(k)}_{1,1}(x)|^2 & \overline{u^{(k)}_{1,1}(x)} u^{(k)}_{1,2}(x)  & \overline{u^{(k)}_{1,1}(x)} u^{(k)}_{1,3}(x) \\
\overline{u^{(k)}_{1,2}(x)} u^{(k)}_{1,1}(x)  & |u^{(k)}_{1,2}(x)|^2& \overline{u^{(k)}_{1,2}(x)} u^{(k)}_{1,1}(x) \\
\overline{u^{(k)}_{1,3}(x)} u^{(k)}_{1,1}(x)  & \overline{u^{(k)}_{1,3}(x)} u^{(k)}_{1,1}(x)  & |u^{(k)}_{1,3}(x)|^2
\end{bmatrix}\\
\begin{bmatrix}
\cos^2(x) & 0 & 0 \\ 0 & \sin^2(x) & 0 \\ 0 &0  & 0
\end{bmatrix} &= \sum^\ell_{k=1} t_k \begin{bmatrix}
|u^{(k)}_{2,1}(x)|^2 & \overline{u^{(k)}_{2,1}(x)} u^{(k)}_{2,2}(x)  & \overline{u^{(k)}_{2,1}(x)} u^{(k)}_{2,3}(x) \\
\overline{u^{(k)}_{2,2}(x)} u^{(k)}_{2,1}(x)  & |u^{(k)}_{2,2}(x)|^2& \overline{u^{(k)}_{2,2}(x)} u^{(k)}_{2,1}(x) \\
\overline{u^{(k)}_{2,3}(x)} u^{(k)}_{2,1}(x)  & \overline{u^{(k)}_{2,3}(x)} u^{(k)}_{1,1}(x)  & |u^{(k)}_{2,3}(x)|^2
\end{bmatrix}
\end{align*}
for all $x \in [-1,1]$.  Thus
\[
X(x) = \begin{bmatrix} \sin^2(x) & \cos^2(x) & 0 \\ 0 & \sin^2(x) & \cos^2(x) \\ \cos^2(x) & 0 & \sin^2(x) \end{bmatrix} = \sum^\ell_{k=1}t_k \left[|u_{j,i}^{(k)}(x)|^2   \right]
\]
for all $x\in [-1,1]$ where the first two columns come from above, and the third column comes as $X(x)$ is doubly stochastic and the right-hand side is a convex combination of unistochastic matrices and thus doubly stochastic.  Hence $X \in \conv(C([-1,1], US_3))$ thereby contradicting Proposition \ref{prop:counter-birkhoff-von-neumann}.  Thus $\vec{A} \notin \conv( \U(\vec{B}))$.
\end{proof}

By analyzing the proofs of Proposition \ref{prop:counter-birkhoff-von-neumann} and Corollary \ref{cor:cont-major-does-not-imply-exact-convex-hull}, only a single point of continuity is required due to the finite number of terms in a convex combination.  Thus the following result also holds.

\begin{cor}\label{cor:small-X-not-convex}
With $\X = \{0\} \cup \{ \frac{1}{n} \, \mid \, n \in \bN\}$, there exists CSD families $\vec{A} = (A_1, A_2)$ and $\vec{B} = (B_1, B_2)$ in $C(\X, \M_3)$ such that $\vec{A} \prec_c \vec{B}$ yet $\vec{A} \notin \conv(\U(\vec{B}))$.
\end{cor}

In regards to Proposition \ref{prop:elementary-equivalences}, it still remains to discuss whether $\vec{A} \in \cconv(\U(\vec{B}))$ implies $\vec{A} \in \conv(\U(\vec{B}))$ and whether $\vec{A} \prec_\pt \vec{B}$ implies $\vec{A} \in \cconv(\U(\vec{B}))$.  The latter will be demonstrated in Section \ref{sec:main} with the proof of Theorem \ref{main_theorem}.  In particular, the CSD families $\vec{A}, \vec{B} \in C([-1,1], \M_3)$ from Corollary \ref{cor:cont-major-does-not-imply-exact-convex-hull} have the property that $\vec{A} \prec_c \vec{B}$ so $\vec{A} \in \cconv(\U(\vec{B}))$ by Proposition \ref{prop:elementary-equivalences} but $\vec{A} \notin \conv(\U(\vec{B}))$ thereby yielding an example where $\vec{A} \in \cconv(\U(\vec{B}))$ does not imply $\vec{A} \in \conv(\U(\vec{B}))$.

\section{Joint Majorization in Finite Dimensional C$^*$-Algebras}
\label{sec:finite-dimensional}

As a stepping stone to proving Theorem \ref{main_theorem}, in this section we will demonstrate that if $\X$ has a finite number of points, then $\vec{A} \prec_{\pt} \vec{B}$ implies $\vec{A} \in \conv(\U(\vec{B}))$.  To demonstrate said result, a characterization of the convex hull of the joint unitary orbit of an abelian family in any finite dimensional C$^*$-algebra is developed.  

The main obstacle to obtaining these results is resolving how one can combine different convex combinations from different portions of the space.  The following algorithm is useful in converting multiple probability vectors into a single probability vector that will work (by repeating some unitaries if necessary).

\begin{lem} \label{finite_algo}
Let $\vec{t}^{(1)}, \dots, \vec{t}^{(k)} \in \bR^\ell$ be probability vectors. Then there exists a $N \in \bN$ and a probability vector $\vec{a} \in \bR^N$ such that for each $1 \leq j \leq k$ there exists a partition $Y_1 \cup \cdots \cup Y_\ell$ of $\{1,\ldots, N\}$ such that
\begin{equation}
\sum_{p \in Y_i} a_p = t_i^{(j)}  \label{eq:decomp}
\end{equation}
for all $1 \leq i \leq \ell$.
\end{lem}
\begin{proof}
Let $N = \ell^k$ and let $\vec{a} \in \bR^N$ be the vector whose entries are
\[
t^{(1)}_{i_1} t^{(2)}_{i_2} \cdots t^{(k)}_{i_k}
\]
for all $1 \leq i_1, i_2, \ldots, i_k \leq \ell$.  Since
\[
\sum^\ell_{i=1} t^{(j)}_{i} = 1
\]
for all $1\leq j \leq k$, it is easily verified that $\vec{a}$ is a probability vector.  Moreover, since for all $1\leq j_0 \leq k$ and $1\leq i_{j_0} \leq \ell$ we have
\[
\sum^k_{\substack{j=1 \\ j \neq j_0}} \sum^\ell_{i_j = 1} t^{(1)}_{i_1} t^{(2)}_{i_2} \cdots t^{(k)}_{i_k} = t^{(j_0)}_{i_{j_0}},
\]
the desired partition that satisfies equation (\ref{eq:decomp}) for $j_0$ exists.
\end{proof}

The use of Lemma \ref{finite_algo} is made clear in the proof of the following result.

\begin{thm} \label{finite_main_thm}
Let $k \in \bN$, let $\fA_1, \ldots, \fA_k$ be unital C$^*$-algebras, and let $\fA = \fA_1 \oplus \cdots \oplus \fA_k$. Let $\vec{A} = (A_1, \dots, A_m)$ and $\vec{B} = (B_1, \dots, B_m)$ be abelian families in $\fA$ and write
\[
A_i = A_i^{(1)} \oplus  \cdots \oplus A_{i}^{(k)} \qand B_i = B_{i}^{(1)} \oplus \cdots \oplus B_i^{(k)}
\]
for all $1\leq i \leq m$.  If for all $1 \leq j \leq k$ we have in $\fA_j$ that
\[
(A_1^{(j)}, \ldots, A_m^{(j)} ) \in \conv\left(\U ( B_1^{(j)}, \ldots, B_m^{(j)} ) \right), 
\]
then $\vec{A} \in \textnormal{conv}(\mathcal{U}(B))$.
\end{thm}
\begin{proof}
By the assumptions, for each $1 \leq j \leq k$, there exists an $\ell_j \in \bN$, unitaries $U_{j, 1}, \dots, U_{j, \ell} \in \fA_j$, and a probability vector $\vec{t}^{(j)} \in \bR^{\ell_j}$ such that
\[
A_i^{(j)} = \sum_{r = 1}^{\ell_j} t_{r}^{(j)} U_{j, r}^* B_i^{(j)}U_{j, r}
\]
for all $1 \leq i \leq m$.  By extending the probability vectors via adding a sufficient number of 0 terms and by adding unitaries equal to the identity, we can assume without loss of generality that $\ell_1 = \ell_2 = \cdots = \ell_m = \ell$. By applying Lemma \ref{finite_algo} to the probability vectors $\vec{t}^{(1)}, \dots, \vec{t}^{(m)} \in \bR^{\ell}$, we obtain a probability vector $\vec{a} \in \bR^N$ for some $N \in \bN$ such that for every $1 \leq j \leq m$  there exists $V_{j,1}, \ldots, V_{j,N} \in \{U_{j,1}, \ldots, U_{j, \ell}\}$ such that
\[
A_i^{(j)} =\sum_{r = 1}^{\ell} t_{r}^{(j)} U_{j,r}^* B_i^{(j)} U_{j,r} = \sum_{q = 1}^N a_qV^*_{j,q} B_i^{(j)}V_{j,q}
\]
for all $1 \leq i \leq m$ (i.e. if $Y_1 \cup \cdots \cup Y_\ell$ is the partition of $\{1, \ldots, N\}$, then $V_{j,q} = U_{j, p}$ whenever $q \in Y_p$).  Consequently, with
\[
V = V_{1,q} \oplus \cdots \oplus V_{m,q}
\]
which is a unitary in $\fA$, it follows that
\begin{align*}
A_i = \left( \sum_{q = 1}^N a_qV^*_{1,q} B_i^{(1)}V_{1,q} \right)  \oplus \cdots \oplus \left(  \sum_{q = 1}^N a_qV^*_{m,q} B_i^{(m)}V_{m,q} \right) = \sum_{q = 1}^N a_q V^*B_i V
\end{align*}
for all $1 \leq i \leq m$.  Thus $\vec{A} \in \conv(\mathcal{U}(\vec{B}))$. 
\end{proof}

\begin{cor} 
Let $\fA$ be a finite dimensional C$^*$-algebra and let $\vec{A} = (A_1, \ldots, A_m)$ and $\vec{B} = (B_1, \ldots, B_m)$ be abelian families in $\fA$.  The following are equivalent:
\begin{enumerate}
\item $\vec{A} \prec_{\tr} \vec{B}$.
\item $\vec{A} \in \conv(\mathcal{U}(B))$.
\item $\vec{A} \in \cconv(\mathcal{U}(B))$.
\end{enumerate}
\end{cor}
\begin{proof}
Recall every finite dimensional C$^*$-algebra is the direct sum of a finite number of matrix algebras.  Hence the result follows immediately from Theorem \ref{thm:matrix-only-result} and Theorem \ref{finite_main_thm}.
\end{proof}

\begin{cor} \label{cor:finte-X-all-good}
Let $k \in \bN$  and let $\X = \{1,\ldots, k\}$.  All of the properties listed in Proposition \ref{prop:elementary-equivalences} are equivalent.
\end{cor}

\section{Joint Continuous Majorization}
\label{sec:main}

In this section, we will provide a proof of Theorem \ref{main_theorem} showing that pointwise joint majorization characterizes the elements of the closed convex hull of the joint unitary orbits of abelian families in continuous matrix algebras.  The same ideas will be used in the next section which characterizes the closed convex hull of the joint unitary orbits of abelian families in subhomogeneous C$^*$-algebras.

In order to prove Theorem \ref{main_theorem}, let $\X$ be a compact metric space and let $\vec{A} = (A_1, \ldots, A_m)$ and $\vec{B} = (B_1,\ldots, B_m)$ be abelian families in $C(\X, \M_n)$.  Since $C(\X)$ is an abelian C$^*$-algebra, the second continuous dual space $\fM = C(\X)^{**}$ is a abelian von Neumann algebra that naturally contains $C(\X)$.  Moreover $C(\X)$ is dense in $\fM$ with respect to the strong operator topology and it is easy to see that the second continuous dual space of $C(\X, \M_n) = C(\X) \otimes \M_n$ is $\fM \otimes \M_n$.

Using the same ideas as in \cite{ng2018majorization}*{Proposition 4.1}, we have the following result which shows that it suffices to prove Theorem \ref{main_theorem} in the context of $\fM \otimes \M_n$.

\begin{lem}
\label{lem:suffices-to-consider-VN-alg}
Using the above notation
\[
\cconv(\U_{C(\X, \M_n)}(\vec{B})) = \cconv (\U_{\fM \otimes \M_n}(\vec{B})) \cap C(\X, \M_n)^m.
\]
\end{lem} 
\begin{proof}
Clearly
\[
\cconv (\U_{C(\X, \M_n)}(\vec{B})) \subseteq \cconv (\U_{\fM \otimes \M_n}(\vec{B})).
\]
To see the other inclusion, let $\vec{B} = (B_1, \ldots, B_m)$ and let 
\[
\vec{C} = (C_1, \ldots, C_m) \in \cconv (\U_{\fM \otimes \M_n}(\vec{B}))  \cap C(\X, \M_n)^m.
\]
Given $\epsilon > 0$, this implies there exists a $k \in \bN$, a probability vector $\vec{t} \in \bR^k$, and unitaries $U_1,\ldots, U_k \in \fM \otimes \M_n$ such that
\[
\left\|(C_1,\ldots, C_m) - \sum^k_{i=1} t_i (U_i^* B_1U_i, \ldots, U^*_i B_mU_i) \right\|_\infty < \epsilon.
\]
Since $C(\X, \M_n)$ is dense in $\fM \otimes \M_n$ with respect to the strong operator topology, the Kaplansky density theorem for unitaries implies for $1 \leq i \leq k$ there exists nets $(U_{i,\lambda})_{\lambda \in \Lambda}$ of unitaries in $C(\X, \M_n)$ that converge in the ultrastrong$^*$ operator topology on $\fM \otimes \M_n$ to $U_i$ respectively.  Hence, by sending $\epsilon$ to zero, we see that the product ultrastrong$^*$-closure of
\[
Y = \left\{\left.  (C_1,\ldots, C_m) - \sum^k_{i=1} t_i (U_{i, \lambda}^* B_1U_{i, \lambda}, \ldots, U^*_{i, \lambda} B_mU_{i, \lambda}) \, \right| \, \lambda \in \Lambda \right\} \subseteq C(\X, \M_n)^m
\]
contains the zero vector $\vec{0}$.

As the ultrastrong$^*$ and ultraweak topologies on $(\fM \otimes \M_n)^m$ have the same continuous linear functionals and thus have the same closed convex sets by the Hahn Banach Theorem, there is a net in $\conv(Y)$ that converges to $\vec{0}$ in the ultraweak topology on $(\fM \otimes \M_n)^m$.  However, since the ultraweak topology on $(\fM \otimes \M_n)^m$ is the weak$^*$-topology generated by $C(\X, \M_n)^*$, since the weak$^*$-topology on $(\fM \otimes \M_n)^m$ generated by $C(\X, \M_n)^*$ when restricted to $C(\X, \M_n)^m$ is the weak topology on $C(\X, \M_n)^m$, and since $Y \subseteq C(\X, \M_n)^m$, there is a net in $\conv(Y)$ that converges to $\vec{0}$ in the weak topology on $C(\X, \M_n)^m$.  Therefore, since the Hahn Banach Theorem implies the weak and norm closures of $\conv(Y)$ in $C(\X, \M_n)^m$ agree, $\vec{0}$ is in the norm closure of $\conv(Y)$ in $C(\X, \M_n)^m$.  Since every element of $\conv(Y)$ is of the form $\vec{C} - \vec{T}$ where $\vec{T} \in \conv(\U_{C(\X, \M_n)}(\vec{B}))$, $\vec{C} \in \cconv (\U_{C(\X, \M_n)}(\vec{B}))$ as desired.
\end{proof}

Since $\fM$ is a abelian von Neumann algebra, $\fM$ is isomorphic to $L_\infty(\Omega, \mu)$ for some measure space $(\Omega, \mu)$.  Consequently $L_\infty(\Omega, \mu) \otimes \M_n$ (which one may view as $\M_n$-valued essentially bounded measurable functions) is the second continuous dual space of $C(\X, \M_n)$.  Using this, the proof of Theorem \ref{main_theorem} can now proceed.

\begin{proof}[Proof of Theorem \ref{main_theorem}]
By Proposition \ref{prop:elementary-equivalences}, it suffices to prove that $\vec{A} \prec_\tr \vec{B}$ implies $\vec{A} \in \cconv (\U_{C(\X, \M_n)}(\vec{B}))$.  Suppose $\vec{A} \prec_\tr \vec{B}$ (using only traces in $C(\X, \M_n)$).  To see that $\vec{A} \in \cconv (\U_{C(\X, \M_n)}(\vec{B}))$ it suffices by Lemma \ref{lem:suffices-to-consider-VN-alg} to show that $\vec{A} \in \cconv (\U_{\fM \otimes \M_n}(\vec{B}))$.  Let $\epsilon > 0$.  

By \cite{K1984}*{Theorem 3.19}, $\{A_1,\ldots, A_m\}$ and $\{B_1,\ldots, B_m\}$ are simultaneously diagonalizable in $L_\infty(\Omega, \mu) \otimes \M_n$.  Hence there exists unitaries $U, V \in L_\infty(\Omega, \mu) \otimes \M_n$ and functions 
\[
\{f_{j, k}, g_{j,k}\, \mid \, 1 \leq j \leq m, 1 \leq k \leq n\} \subseteq L_\infty(\Omega, \mu)
\]
such that
\[
UA_jU^* = \diag(f_{j,1},\ldots, f_{j,n}) \qand VB_jV^* = \diag(g_{j,1},\ldots, g_{j,n})
\]
for all $1 \leq j \leq m$.

Since $C(\X, \M_n) \subseteq L_\infty(\Omega, \mu) \otimes \M_n$, every trace on $L_\infty(\Omega, \mu) \otimes \M_n$ restricts to a trace on $C(\X, \M_n)$.  Consequently, since $\vec{A} \prec_\tr \vec{B}$ (using only traces in $C(\X, \M_n)$) and since only a countable number of continuous convex functions need to be considered in tracial majorization (see Lemma \ref{lem:approx-convex-from-below-using-hyperplanes}), it must be the case that $\vec{A}(y) \prec_\tr \vec{B}(y)$ for almost every $y \in \Omega$.

Without loss of generality, we may assume $\vec{A}(y) \prec_\tr \vec{B}(y)$ for all $y \in \Omega$.  Hence Theorem \ref{thm:matrix-only-result} implies that $\vec{A}(y) \prec \vec{B}(y)$ for all $y\in \Omega$.  Therefore, for all $y \in \Omega$ there exists a $X(y) \in DS_n$ such that
\[
X(y) (g_{j,1}(y), \ldots, g_{j,n}(y))^T = (f_{j,1}(y), \ldots, f_{j,n}(y))^T
\]
for all $1 \leq j \leq m$.  

Since the maps $y \mapsto (g_{j,1}(y), \ldots, g_{j,n}(y))$ and $y \mapsto (f_{j,1}(y), \ldots, f_{j,n}(y))$ are measurable for all $1 \leq j \leq m$ and bounded entrywise by the operator norms of $\vec{A}$ and $\vec{B}$, there exists a partition $\Omega_1 \cup \cdots \cup \Omega_\ell$ of $\Omega$ into pairwise disjoint measurable sets such that for all $1 \leq p \leq \ell$ and $y_1, y_2 \in \Omega_p$ we have
\begin{align*}
\left\| (g_{j,1}(y_1), \ldots, g_{j,n}(y_1)) - (g_{j,1}(y_2), \ldots, g_{j,n}(y_2))\right\|_\infty & < \epsilon \text{ and}\\
\left\| (f_{j,1}(y_1), \ldots, f_{j,n}(y_1)) - (f_{j,1}(y_2), \ldots, f_{j,n}(y_2))\right\|_\infty & < \epsilon
\end{align*}
for all $1 \leq j \leq m$.  

For each $1 \leq p \leq \ell$, choose a $y_p \in \Omega_p$.  Since  $X(y_p)$ is doubly stochastic, by Theorem \ref{thm:matrix-only-result} there exists a probability vector $\vec{t}_p \in \bR^{k_p}$ and unitary matrices $W_{p, 1}, \ldots, W_{p, k_p} \in \M_n$ such that
\begin{align*}
&\left\|\diag(f_{j,1}(y_p), \ldots, f_{j,n}(y_p)) -  \sum^{k_p}_{i=1} t_{p,i} W_{p,i}^* \diag(g_{j,1}(y_p), \ldots, g_{j,n}(y_p)) W_{p,i}\right\|< \epsilon.
\end{align*}
Hence, by simple norm approximations, it is elementary to see that
\begin{align*}
&\left\|\diag(f_{j,1}(y), \ldots, f_{j,n}(y)) -  \sum^{k_p}_{i=1} t_{p,i} W_{p,i}^* \diag(g_{j,1}(y), \ldots, g_{j,n}(y)) W_{p,i}\right\|< 3\epsilon
\end{align*}
for all $y \in \Omega_p$ and $1 \leq j \leq m$ (i.e. conjugation by unitaries preserves the norm and the norm of the convex combination is bounded above by the convex combination of the norms).  Moreover, by the same arguments as used in the proof of Theorem \ref{finite_main_thm}, we may assume without loss of generality that $k_1 = \ldots = k_\ell = k$ and $\vec{t}_1 = \cdots = \vec{t}_\ell = \vec{t} \in \bR^k$.

For each $1 \leq i \leq k$, let $W_i : \Omega \to \M_n$ be defined by
\[
W_i(y) = W_{p,i} \quad \text{for all }y \in \Omega_p.
\]
Clearly $W_i \in L_\infty(\Omega, \mu) \otimes \M_n$ as each $\Omega_p$ is measurable and $W_i$ is a unitary for all $1 \leq i \leq k$.  Moreover
\[
\left\|\diag(f_{j,1}(y), \ldots, f_{j,n}(y)) -  \sum^{k}_{i=1} t_i W_{i}^* \diag(g_{j,1}(y), \ldots, g_{j,n}(y)) W_{i}\right\| \leq 2\epsilon
\]
for all $y \in \Omega$ and $1 \leq j \leq m$.  Hence we see that
\[
\left\| A_j - \sum^{k}_{i=1} t_i (V^* W_{i}U)^*  B_j (V^* W_{i}U)\right\| \leq 2\epsilon
\]
for all $1 \leq j \leq m$.  Thus $\vec{A} \in \cconv (\U_{\fM \otimes \M_n}(\vec{B}))$.
\end{proof}

\section{Joint Majorization in Subhomogeneous C$^*$-Algebras}
\label{sec:subhomo}

Recall a  C$^*$-algebra $\fA$ is subhomogeneous if there exists an $N \in \bN$ and a faithful representation of $\fA$ that consists of a direct sum of representations onto matrix algebras of size at more $N$.  In particular, for a compact metric space $\X$, $C(\X, \M_N)$ is a subhomogeneous C$^*$-algebra.

The goal of this section is to extend Theorem \ref{main_theorem} to subhomogeneous C$^*$-algebras.  Of course pointwise doubly stochastic majorization does not make sense in this context.  However, tracial majorization does and is equivalent to pointwise doubly stochastic majorization for $C(\X, \M_n)$.  As the proof that $\vec{A} \in \cconv(\U(\vec{B}))$ implies $\vec{A} \prec_\tr \vec{B}$ in Proposition \ref{prop:elementary-equivalences} makes use of doubly stochastic matrices which do not exist in this context, an alternative proof that works for every C$^*$-algebra is the first order of business.  First we reduce the number of continuous convex functions that need to be considered in tracial majorization.

\begin{lem}
\label{lem:approx-convex-from-below-using-hyperplanes}
Let $K$ be a compact convex subset of $\bR^m$ and let $f : K \to [0, \infty)$ be a continuous convex function.  For all $\epsilon > 0$ there exists an $\ell \in \bN$ and functions $g_1, \ldots, g_\ell : \bR^m \to \bR$ of the form
\[
g_i(x_1, \ldots, x_m) = a_{i,1} x_1 + \cdots + a_{i,m} x_m + b_j
\]
for $a_{i,}, \ldots, a_{i,m}, b_i \in \bR$ such that if $g : K \to \bR$ is defined by
\[
g(\vec{x}) = \max\{0, g_1(\vec{x}), \ldots, g_\ell(\vec{x})\},
\]
then $f-\epsilon \leq g  \leq f $ on $K$.
\end{lem}
\begin{proof}
Since $f$ is convex and continuous, the set
\[
C = \{(x_1, \ldots, x_m,y) \in \bR^{m+1} \, \mid \, (x_1, \ldots, x_m) \in K, y \geq  f(x_1, \ldots, x_m)\}
\]
is a closed convex set.  The Hahn-Banach Theorem implies for each $\vec{y} \in K$ there exists $a_{\vec{y},1}, \ldots, a_{\vec{y},m}, b_{\vec{y}} \in \bR$ such that if $g_{\vec{y}} : K \to \bR$ is defined by
\[
g_{\vec{y}}(x_1, \ldots, x_m) = a_{\vec{y},1} x_1 +  \cdots +  a_{\vec{y},m}x_m + b_{\vec{y}}
\]
then $g_{\vec{y}} \leq f$ on $K$ and $f(\vec{y}) \leq g_{\vec{y}}(\vec{y}) + \frac{\epsilon}{2}$.  As both $f$ and $g_{\vec{y}}$ are continuous, there exists an open neighbourhood $U_{\vec{y}}$ of $\vec{y}$ such that $f(\vec{z}) \leq g_{\vec{y}}(\vec{z}) + \epsilon$ for all $\vec{z} \in U_{\vec{y}}$.  The compactness of $K$ then yields the result.
\end{proof}

The following is the essential part in proving $\vec{A} \in \cconv(\U(\vec{B}))$ implies $\vec{A} \prec_{\tr} \vec{B}$ in any unital C$^*$-algebra.

\begin{lem}
\label{lem:max-lemma}
Let $\fA$ be a unital C$^*$-algebra and let $\vec{A} = (A_1, \ldots, A_m)$ and $\vec{B} = (B_1, \ldots, B_m)$ be abelian families in $\fA$.  If $\vec{A} \in \cconv(\U(\vec{B}))$ then
\[
\tau(\max\{0, A_1, \ldots, A_m\}) \leq \tau(\max\{0, B_1, \ldots, B_m\}).
\]
for all tracial states $\tau$ on $\fA$.
\end{lem}
\begin{proof}
Fix a tracial state $\tau$ on $\fA$.  Let $\A = C^*(A_1,\ldots, A_n) \subseteq \fA$ and let
\[
A' = \max\{0, A_1, \ldots, A_m\} \qqand B' = \max\{0, B_1, \ldots, B_m\} .
\]
Note $\A$ is an abelian C$^*$-algebra and there is an isomorphism $\pi : \A \to C(K)$ where $K \subseteq \bR^m$ is a compact subset and $\pi(A_j) = x_j$ (the projection onto the $j^{\mathrm{th}}$ coordinate) for all $1 \leq j \leq m$.

To see that $\tau(A') \leq \tau(B')$, let $\epsilon > 0$.  Since $\vec{A} \in \cconv(\U(\vec{B}))$ there exists an $\ell \in \bN$, a probability vector $\vec{t} \in \bR^\ell$, and unitaries $U_1, \ldots, U_\ell \in \U(\fA)$ such that
\[
\left\|A_j - \sum^\ell_{k=1} t_k U^*_k B_j U_k\right\| < \epsilon
\]
for all $1 \leq j \leq m$.  Let
\[
C = \sum^\ell_{k=1} t_k U^*_k B' U_k.
\]
Thus $C \geq 0$ (as $B' \geq 0$),
\[
\tau(C) = \sum^\ell_{k=1} t_k \tau(U^*B'U) = \sum^\ell_{k=1} t_k \tau(B') = \tau(B'),
\]
and
\begin{align*}
A_j - \epsilon I \leq \sum^\ell_{k=1} t_k U^*_k B_j U_k \leq \sum^\ell_{k=1} t_k U^*_k B' U_k = C
\end{align*}
for all $1 \leq j \leq m$.

Since $\A$ is isomorphic to $C(K)$, for each $1 \leq j \leq m$ there exists a positive contraction $S_j \in \A$ such that
\[
S_j (A_j - \epsilon I) = \max\left\{A_j - 2\epsilon, 0\right\}.
\]
For each $1 \leq j \leq m$ let
\[
V_j = \left\{(x_1, \ldots, x_m) \in K \, \left| \, x_j - 2 \epsilon > \max\{x_1, \ldots, x_m\} - 3\epsilon\right.\right\}.
\]
Clearly $\{V_j\}^m_{j=1}$ is a open cover of $K$ so there exists a partition of unity $\{f_j\}^m_{j=1}$ of $K$ subordinate to the covering $\{V_j\}^m_{j=1}$.  Thus if 
\[
T_j = f_j( \max\{0, A' - 3\epsilon\})
\]
for all $1 \leq j \leq m$, then $T_j \geq 0$ for all $1 \leq j \leq m$ and
\[
\sum^m_{j=1}T_j = I.
\]
Moreover, notice by the definition of $V_j$ that
\[
T_j^\frac{1}{2}  \max\{0, A' - 3\epsilon\} T_j^\frac{1}{2} \leq T_j^\frac{1}{2} \max\left\{0, A_j- 2\epsilon\right\} T_j^\frac{1}{2} = T_j^\frac{1}{2} S_j^\frac{1}{2} (A_j - \epsilon I) S_j^\frac{1}{2} T_j^\frac{1}{2}
\]
for all $1\leq j \leq m$.  Therefore, since $A_j - \epsilon I \leq C$ implies
\[
T_j^\frac{1}{2} S_j^\frac{1}{2} (A_j - \epsilon I) S_j^\frac{1}{2} T_j^\frac{1}{2} = S_j^\frac{1}{2} T_j^\frac{1}{2} (A_j - \epsilon) T_j^\frac{1}{2} S_j^\frac{1}{2} \leq S_j^\frac{1}{2} T_j^\frac{1}{2} C T_j^\frac{1}{2} S_j^\frac{1}{2}
\]
for all $1 \leq j \leq m$, we obtain that
\begin{align*}
\tau(\max\{0, A' -3 \epsilon\}) & = \sum^m_{j=1} \tau \left( T_j^\frac{1}{2}  \max\{0, A' - 3\epsilon\} T_j^\frac{1}{2}  \right) \\
& \leq \sum^m_{j=1} \tau\left(S_j^\frac{1}{2} T_j^\frac{1}{2} C T_j^\frac{1}{2} S_j^\frac{1}{2}\right) \\
& = \sum^m_{j=1} \tau\left(C^\frac{1}{2} T_j^\frac{1}{2} S_j T_j^\frac{1}{2} C^\frac{1}{2}\right) \\
& \leq \sum^m_{j=1} \tau\left(C^\frac{1}{2} T_j C^\frac{1}{2}\right) \\
&= \tau(C) = \tau(B').
\end{align*}
Therefore, as $\max\{0, A' -3 \epsilon\}$ increases to $A'$ in norm as $\epsilon$ tends to 0, the continuity of $\tau$ yield the result.
\end{proof}

\begin{thm}
\label{thm:joint-major-implies-tracial-major}
Let $\fA$ be a unital C$^*$-algebra and let $\vec{A} = (A_1, \ldots, A_m)$ and $\vec{B} = (B_1, \ldots, B_m)$ be abelian families in $\fA$.  If  $\vec{A} \in \cconv(\U(\vec{B}))$ then $\vec{A} \prec_\tr \vec{B}$.
\end{thm}
\begin{proof}
Suppose  $\vec{A} \in \cconv(\U(\vec{B}))$.
Note if $[a_{i,j}] \in \M_{\ell \times m}$ and we define
\begin{align*}
\vec{A}' &= (a_{1,1} A_1 + \cdots + a_{1,m} A_m, \ldots, a_{\ell,1} A_1 + \cdots + a_{\ell,m} A_m) \text{ and}\\
\vec{B}' &= (a_{1,1} B_1 + \cdots + a_{1,m} B_m, \ldots, a_{\ell,1} B_1 + \cdots + a_{\ell,m} B_m),
\end{align*}
then $\vec{A}'$ and $\vec{B}'$ are abelian families in $\fA$ such that  $\vec{A}' \in \cconv(\U(\vec{B}'))$.  Hence Lemma \ref{lem:max-lemma} implies that
\[
\tau(g(\vec{A})) \leq \tau(g(\vec{B}))
\]
for all tracial states $\tau$ and all functions $g$ as in the conclusions of Lemma \ref{lem:approx-convex-from-below-using-hyperplanes}.  Hence Lemma \ref{lem:approx-convex-from-below-using-hyperplanes} (along with the fact that only continuous convex functions on the combined joint spectra need to be considered) completes the proof.
\end{proof}

Using Theorem \ref{thm:joint-major-implies-tracial-major} and the same ideas as used in the proof of Theorem \ref{main_theorem}, we have the following result for subhomogeneous C$^*$-subalgebras.

\begin{thm}\label{thm:subhomo}
Let $\fA$ be a separable subhomogeneous C$^*$-subalgebra.  Suppose $\vec{A} = (A_1, \ldots, A_m)$ and $\vec{B} = (B_1, \ldots, B_m)$ are abelian families in $\fA$.  Then the following are equivalent:
\begin{enumerate}
\item $\vec{A} \prec_\tr \vec{B}$.
\item $\vec{A} \in \cconv(\U(\vec{B}))$.
\end{enumerate}
\end{thm} 
\begin{proof}
Note $\vec{A} \in \cconv(\U(\vec{B}))$ implies $\vec{A} \prec_\tr \vec{B}$ by Theorem \ref{thm:joint-major-implies-tracial-major}.

Conversely, suppose $\vec{A} \prec_\tr \vec{B}$.  Since $\fA$ is a subhomogeneous C$^*$-subalgebra, there exists an $N \in \bN$ and a faithful representation of $\fA$ onto $\B(\H)$ for a separable $\H$ that consists of a direct sum of representations onto matrix algebras of size at more $N$.  The von Neumann algebra completion $\fM$ of $\fA$ in the image of this representation is then of the form
\[
\fM = (L_\infty(\Omega_1, \mu_1) \otimes \M_1) \oplus \cdots \oplus (L_\infty(\Omega_N, \mu_N) \otimes \M_N)
\]
for some (possibly empty) measure spaces $(\Omega_k, \mu_k)$ for $1\leq k \leq N$.  

By the same proof as used in Lemma \ref{lem:suffices-to-consider-VN-alg}, it suffices to prove that
\[
\vec{A} \in \cconv(\U_\fM(\vec{B})).
\]
Moreover, since $\vec{A} \prec_\tr \vec{B}$ in $\fA$, it follows that $\vec{A} \prec_\tr \vec{B}$ in $\fM$.  If for $1 \leq k \leq N$ the map $\pi_k : \fM \to  L_\infty(\Omega_k, \mu_k) \otimes \M_k$ denotes the projection onto the $k^{\mathrm{th}}$ component in the direct sum of $\fM$, it follows that $\pi_k(\vec{A}) \prec_\tr \pi_k(\vec{B})$ and the proof of Theorem \ref{main_theorem} yields
\[
\pi_k(\vec{A}) \in \cconv(\U_{ L_\infty(\Omega_k, \mu_k) \otimes \M_k}(\pi_k(\vec{B})))
\]
for all $1 \leq k \leq N$.  Thus, using Lemma \ref{lem:nice-prob-vector} to correctly divide the approximating convex combinations from the finitely many different components of $\fM$ into those of equal length and equal probability vectors as was done in Theorem \ref{finite_main_thm}, it follows that $\vec{A} \in \cconv(\U_\fM(\vec{B}))$ thereby yielding the result.
\end{proof}

\section{Additional Examples and Questions}
\label{sec:bonus}

Returning to the discussion of majorization in $C(\X, \M_n)$, there are a number of interesting (and likely difficult) questions that still need resolving.  For example:

\begin{ques}\label{ques:exact-for-all}
For which compact Hausdorff spaces $\X$ is it true that if $\vec{A} = (A_1,\ldots, A_m)$ and $\vec{B} = (B_1, \ldots, B_m)$ are CSD in $C(\X, \M_n)$ such that $\vec{A} \prec_c \vec{B}$ then $\vec{A} \in \conv(\U(\vec{B}))$?
\end{ques}

Recall if $\X$ has a finite number of points then $\vec{A} \prec_c \vec{B}$ implies $\vec{A} \in \conv(\U(\vec{B}))$ by Corollary \ref{cor:finte-X-all-good}, but Corollaries \ref{cor:cont-major-does-not-imply-exact-convex-hull} and \ref{cor:small-X-not-convex} give two examples of such $\X$ where $\vec{A} \prec_c \vec{B}$ does not imply $\vec{A} \in \conv(\U(\vec{B}))$.  However, these latter examples are not sub-Stonean.  This is of interest because \cite{GP1984} proved that abelian families in $C(\X, \M_n)$ are always CSD if and only if $\X$ is sub-Stonean with $\dim(\X) \leq 2$ and carries no non-trivial $G$-bundles over any closed subset for $G$ a symmetric group or the circle group.  It would be quite interesting if the answer to Question \ref{ques:exact-for-all} was the same condition.

Although Corollaries \ref{cor:cont-major-does-not-imply-exact-convex-hull} and \ref{cor:small-X-not-convex} demonstrate that $\vec{A} \prec_c \vec{B}$ does not $\vec{A} \in \conv(\U(\vec{B}))$ when $\vec{A}$ and $\vec{B}$ have length (greater than or equal to) 2, we were unable to find an example for single operator majorization.  In particular, the following question is still open:

\begin{ques}\label{ques:m=1?}
Let $A, B \in C(\X, \M_n)$ be self-adjoint and continuously diagonalizable.  Does $A \prec_c B$ imply that $A \in \conv(\U(B))$?
\end{ques}

Of course there are a number of barriers to overcome for an affirmative answer to Question \ref{ques:m=1?}: unitary orbits in infinite dimensional C$^*$-algebras are often not closed; the operators here need not be continuously diagonalizable unless $\X$ is sub-Stonean etc.; the number of unitaries used and probability vectors need not be well-behaved.  In terms of the last condition, note the technique developed in Lemma \ref{lem:nice-prob-vector} required finiteness and issues with the infinite number of probability vectors one needs to consider was the core of the issue in Corollary \ref{cor:cont-major-does-not-imply-exact-convex-hull}. However, as the following two results show, some of these might not be issues.

\begin{lem}\label{lem:nice-prob-vector}
Let $A, B \in \M_n$ be self-adjoint operators such that $A \prec B$.  Then there exists unitaries $U_1, \ldots, U_{2^{n-1}} \in \M_n$ such that
\[
A = \sum^{2^{n-1}}_{i=1} \frac{1}{2^{n-1}} U^*_i B U_i.
\]
\end{lem}
\begin{proof}
Since $A$ and $B$ are self-adjoint, there exist unitary operators $U,V \in \M_n$ such that
\[
A = U^* \diag(a_1, \ldots, a_n)U \qand B = V^* \diag(b_1, \ldots, b_n) V
\]
where $\vec{a} = (a_1, \ldots, a_n) \in \bR^n$ and $\vec{b} = (b_1, \ldots, b_n) \in \bR^n$.  By \cite{A1989}*{Theorem 1.3} (specifically the proof) there exists a sequence $\vec{c}_0, \vec{c}_1, \ldots, \vec{c}_{n-1} \in \bR^n$ such that
\begin{itemize}
\item $\vec{c}_0 = \vec{b}$,
\item $\vec{c}_{n-1} = \vec{a}$,
\item at least $k$ entries of $\vec{c}_k$ equal the corresponding entries of $\vec{a}$ for all $0 \leq k \leq n-1$, 
\item for any $1 \leq k \leq n-2$, $\vec{c}_k$ and $\vec{c}_{k+1}$ differ in either 0 or 2 entries, and
\item for any $1 \leq k \leq n-2$, if $\vec{c}_k$ and $\vec{c}_{k+1}$ differ in the entries $1 \leq i,j \leq n$ with $i \neq j$, then there exists a $t_k \in (0,1)$ such that $c_{k+1, i} = t_k c_{k,i} + (1-t_k) c_{k,j}$ and $c_{k+1, j} = (1-t_k) c_{k,i} + t_k c_{k,j}$.
\end{itemize}
In the case where $\vec{c}_k$ and $\vec{c}_{k+1}$ differ in two entries as above, define $W_{k,1}, W_{k,2} \in \M_n$ by
\begin{align*}
W_{1,k} e_p &= e_p \text{ if }p \neq i,j  & W_{2,k} e_p &= e_p \text{ if }p \neq i,j \\
W_{1,k} e_i &= \cos(\theta) e_i + \sin(\theta) e_j & W_{2,k} e_i &= \cos(\theta) e_i - \sin(\theta) e_j\\
W_{1,k} e_j &= -\sin(\theta) e_i + \cos(\theta) e_j & W_{2,k} e_j &= \sin(\theta) e_i + \cos(\theta) e_j
\end{align*}
where $\theta \in \bR$ is such that $\cos^2(\theta) = t$.  Clearly $W_{1,k}$ and $W_{2,k}$ are unitaries.  Moreover, a computation shows that
\[
\diag(\vec{c}_{k+1}) = \frac{1}{2} W_{1,k}^* \diag(\vec{c}_k) W_{1,k} + \frac{1}{2} W_{2,k}^* \diag(\vec{c}_k) W_{2,k}.
\]
Hence, by substituting one convex combination into another and by repeating unitaries if necessary, there exists unitaries $W_1, \ldots, W_{2^{n-1}} \in \M_n$ such that
\[
\diag(a_1, \ldots, a_n) = \sum^{2^{n-1}}_{i=1} \frac{1}{2^{n-1}} W^*_i \diag(b_1, \ldots, b_n) W_i.
\]
Hence $U_i = VW_i U^*$ satisfy the desired equation.
\end{proof}

It is worthwhile to note that \cite{BL2020}*{Proposition 2.6} shows one can replace the occurrences of $2^{n-1}$ in Lemma \ref{lem:nice-prob-vector} with $n$ via the Schur-Horn Theorem.  The above proof has been included for discussions at the end of this section.

\begin{thm}\label{thm:everything-good-in-infinite-dimensions?}
Let $A, B \in \ell_\infty(\bN, \M_n)$ be self-adjoint.  Then $A \prec_c B$ if and only if $A \in \conv(\U(B))$.
\end{thm}
\begin{proof}
Recall if $A \in \conv(\U(B))$ then $A \prec_c B$ by Proposition \ref{prop:elementary-equivalences}.

Conversely, suppose $A \prec_c B$.  Write $A = (A_k)_{k\geq 1}$ and $B = (B_k)_{k\geq 1}$ for self-adjoint $A_k, B_k \in \M_n$.  Since $A \prec_c B$ implies $A \prec_\pt B$ by Proposition \ref{prop:elementary-equivalences}, Theorem \ref{thm:matrix-only-result} implies $A_k \prec B_k$ for all $k \in \bN$.  Hence Lemma \ref{lem:nice-prob-vector} implies for all $k \in \bN$ there exist unitaries $U_{k,1},\ldots, U_{k,2^{n-1}} \in \M_n$ such that
\[
A_k = \sum^{2^{n-1}}_{i=1} \frac{1}{2^{n-1}} U_{k,i}^* B_k U_{k,i}.
\]
Therefore, for all $1 \leq i \leq 2^{n-1}$ we define $U_i = (U_{k,i})_{k\geq 1}$, then $U_i \in \ell_\infty(\bN, \M_n)$ is a unitary such that
\[
A = \sum^{2^{n-1}}_{i=1} \frac{1}{2^{n-1}} U^*_i B U_i. \qedhere
\]
\end{proof}

The proof of Theorem \ref{thm:everything-good-in-infinite-dimensions?} is possible because although doubly stochastic matrices used in joint majorization need not be unistochastic, if we only care about majorization of single self-adjoint operators, \cite{A1989}*{proof of Theorem 1.4} (which is effectively what is used in Lemma \ref{lem:nice-prob-vector}) says one can use a unistochastic matrix that is obtained by an product of $2 \times 2$ rotation matrices thereby yielding the result.  Consequently, if we let $SDS_n$ denote all of the doubly stochastic matrices $X$ where there exist a $t \in [0,1]$ and $1 \leq i,j \leq n$ with $i \neq j$ such that
\begin{align*}
Xe_k &= e_k \text{ if } k \neq i,j \\
Xe_i &= t e_i  + (1-t) e_j \\
X e_j &= (1-t) e_i + t e_j
\end{align*}
where $\{e_m\}^n_{m=1}$ is the standard basis of $\bR^n$, an affirmative answer to the following question yields an affirmative answer to Question \ref{ques:m=1?}.

\begin{ques}\label{ques:special-ds}
If $\X$ is a compact metric space, if $f_1,\ldots, f_n, g_1, \ldots, g_n \in C(\X)$, and if $X \in C(\X, DS_n)$ are such that
\[
X(x) (g_1(x), \ldots, g_n(x))^T = (f_1(x), \ldots, f_n(x))^T
\]
for all $x \in \X$, does there exists a $Y \in C(\X, DS_n)$ that is a product of elements of $C(\X, SDS_n)$ such that
\[
Y(x) (g_1(x), \ldots, g_n(x))^T = (f_1(x), \ldots, f_n(x))^T
\]
for all $x \in \X$?
\end{ques}

An affirmative answer to Question \ref{ques:special-ds} via an algorithmic procedure like that described in Lemma \ref{lem:nice-prob-vector} should yield an appropriate element of $C(\X, US_n)$ to develop a continuous version of the Schur-Horn Theorem.

\section*{Acknowledgements}

We would like to thank the referee of this paper for their useful comments, corrections, and improvements such as the updated example used in Proposition \ref{prop:point-doesnt-go-to-continuous}.


\end{document}